\documentclass[12pt,reqno]{amsart}

\usepackage{graphicx,subfigure}
\usepackage{hyperref}
\hypersetup{colorlinks=true, citecolor=blue, linkcolor=red}

\numberwithin{equation}{section} \numberwithin{figure}{section}
\numberwithin{table}{section} 
{ \theoremstyle{plain}

}

{ \theoremstyle{definition}
\newtheorem{thm}{Theorem}

\numberwithin{equation}{section} \numberwithin{lem}{section}
\numberwithin{thm}{section} \numberwithin{cor}{section}
\numberwithin{pro}{section} \numberwithin{rem}{section}
}

\begin{document}

\title[A Liouville theorem for the vectorial Allen-Cahn energy]
{A Liouville theorem for minimizers with finite potential energy for the vectorial Allen-Cahn equation}



\author{Christos Sourdis} \address{Department of Mathematics and Applied Mathematics, University of
Crete.}
              \email{csourdis@tem.uoc.gr}           




\maketitle

\begin{abstract}
We prove that if a globally minimizing solution to the vectorial
Allen-Cahn equation has finite potential energy, then it is a
constant.
\end{abstract}


Consider the semilinear elliptic system
\begin{equation}\label{eqEq}
\Delta u=\nabla W(u)\ \ \textrm{in}\ \ \mathbb{R}^n,\ \ n\geq 1,
\end{equation}
where $W:\mathbb{R}^m\to \mathbb{R}$, $m\geq 1$, is sufficiently
smooth and nonnegative. It has been recently shown in
\cite{alikakosBasicFacts} that  each nonconstant solution to the
system (\ref{eqEq}) satisfies:
\begin{equation}\label{eqGrande}
\int_{B_R}^{}\left\{\frac{1}{2}|\nabla u|^2+ W\left(u\right)
\right\}dx\geq \left\{\begin{array}{ll}
                        c R^{n-2} & \textrm{if}\ n\geq 3,   \\
                          &     \\
                        c \ln R  & \textrm{if}\ n=2,
                      \end{array}
\right.
\end{equation}
for all $R>1$, and some $c>0$, where $B_R$ stands for the
$n$-dimensional ball of radius $R$, centered at the origin.

On the other side, if additionally $W$ vanishes at least at one
point, it is easy to see that \begin{equation}\label{equpper}
\int_{B_R}^{}\left\{\frac{1}{2}|\nabla u|^2+ W\left(u\right)
\right\}dx\leq CR^{n-1},\ \ R>0,
\end{equation}
for some $C>0$ (see \cite{fuscoPreprint}).

 The
system (\ref{eqEq}) with $W\geq 0$ vanishing at a finite number of
global minima (typically nondegenerate), and coercive at infinity,
is used to model multi-phase transitions (see \cite{fuscoPreprint}
and the references therein). In this case, the system (\ref{eqEq})
is frequently referred to as the vectorial Allen-Cahn equation. In
\cite{sourdis14}, we showed the following theorem for globally
minimizing solutions (see \cite{fuscoCPAA,sourdis14} for the precise
definition).
\begin{thm}\label{thmMine} Assume that $W\in
C^1(\mathbb{R}^m;\mathbb{R})$, $m\geq 1$, and that there exist
finitely many $N\geq 1$ points $a_i\in \mathbb{R}^m$ such that
\begin{equation}\label{eqpoints}
W(u)>0\ \ \textrm{in}\ \ \mathbb{R}^m\setminus \{a_1,\cdots, a_N \},
\end{equation}
and there exists small $r_0>0$ such that the functions
\begin{equation}\label{eqmonot}
r\mapsto W(a_i+r\nu),\ \ \textrm{where}\ \ \nu \in \mathbb{S}^1, \ \
\textrm{are strictly increasing\ for}\ r\in (0,r_0),\ \
i=1,\cdots,N.
\end{equation}
Moreover, we assume that
\begin{equation}\label{eqinf}\liminf_{|u|\to \infty}
W(u)>0.\end{equation}

 If $u\in C^2(\mathbb{R}^2;\mathbb{R}^m)$ is
a bounded, nonconstant, and globally minimizing solution to the
elliptic system (\ref{eqEq}) with $n=2$, there exist   constants
$c_0, R_0>0$ such that
\[
\int_{B_R}^{}W\left(u(x) \right) dx\geq c_0R\ \ \textrm{for}\ \
R\geq R_0.
\]
\end{thm}

In view of (\ref{equpper}), the above result captures the optimal
growth rate in the case $n=2$. The purpose of this note is to
establish the following Liouville type theorem which  holds in any
dimension. Similarly to \cite{sourdis14}, we combine ideas from the
study of vortices in the Ginzburg-Landau model \cite{bethuel} with
variational maximum principles from the study of the vector
Allen-Cahn equation \cite{alikakosPreprint}.
\begin{thm}\label{thmNote}
Let $W$ be as in Theorem \ref{thmMine}. Suppose that $u\in
C^2(\mathbb{R}^n;\mathbb{R}^m)$, $n\geq 2$, is a bounded and
globally minimizing solution to the elliptic system (\ref{eqEq})
such that
\[
 \int_{\mathbb{R}^n}^{}W\left(u(x) \right) dx<\infty.
\]
Then, we have that
\[
u\equiv a_i\ \ \textrm{for some}\ i\in \{1,\cdots,N \}.
\]

\end{thm}

\begin{proof}
It follows that there exists a constant $C_0>0$ such that
\begin{equation}\label{eqj1}
\int_{B_R}^{}W\left(u(x) \right) dx\leq C_0,\ \ R>0.
\end{equation}
Let
\[
\varepsilon=\frac{1}{R}\ \ \textrm{and}\ \
u_\varepsilon(y)=u\left(\frac{y}{\varepsilon} \right),\ \ y\in B_1.
\]
Then, relation (\ref{eqj1}) becomes
\begin{equation}\label{eqj2}
\int_{B_1}^{}W\left(u_\varepsilon(y) \right) dy\leq
C_1\varepsilon^n,\ \ \varepsilon>0,
\end{equation}
for some $C_1>0$. Note that, by standard elliptic regularity
estimates \cite{Gilbarg-Trudinger}, we have that
\begin{equation}\label{eqj3}
|u_\varepsilon|+\varepsilon|\nabla u_\varepsilon|\leq C_2\ \
\textrm{in}\ \mathbb{R}^n, \ \varepsilon>0,
\end{equation}
for some $C_2>0$.

Let $d>0$ be any small number. As in \cite{bethuel}, by combining
(\ref{eqj2}) and (\ref{eqj3}), we deduce that the set where
$W(u_\varepsilon)$ is above $d>0$ is included in a uniformly bounded
number of balls of radius $\varepsilon$, as $\varepsilon\to 0$.
Certainly, there exists $r_\varepsilon \in
(\frac{1}{4},\frac{3}{4})$ such that
\[
W\left(u_\varepsilon(y) \right)\leq d\ \ \textrm{if}\
|y|=r_\varepsilon.
\]
Since $d>0$ is arbitrary, we are led to $\tilde{r}_\varepsilon \in
(\frac{1}{4},\frac{3}{4})$ such that
\[
\max_{|y|=\tilde{r}_\varepsilon}W\left(u_\varepsilon(y) \right)\to
0\ \ \textrm{as}\ \varepsilon\to 0.
\]
In terms of $u$ and $R$, we have
\[
\max_{|x|=s_R}W\left(u(x) \right)\to 0\ \ \textrm{as}\ R\to \infty,\
\ \textrm{for some}\ s_R \in \left(\frac{1}{4}R,\frac{3}{4}R\right).
\]

In view of (\ref{eqinf}), the above relation implies that there
exist $i_j\in \{1,\cdots, N \}$ such that
\[
\max_{|x|=s_R}\left|u(x)-a_{i_j} \right|\to 0\ \ \textrm{as}\ R\to
\infty.
\]
By virtue of (\ref{eqmonot}),  as in \cite{sourdis14}, exploiting
the fact that $u$ is a globally minimizing solution, we can apply a
recent variational maximum principle from \cite{alikakosPreprint} to
deduce that
\[
\max_{|x|\leq s_R}\left|u(x)-a_{i_j} \right|\leq
\max_{|x|=s_R}\left|u(x)-a_{i_j} \right|\ \ \textrm{for}  \ \ R\gg
1.
\]
 The above two
relations imply the existence of an $i_0\in \{1,\cdots,N\}$ such
that
\[
\max_{|x|\leq s_R}\left|u(x)-a_{i_0} \right|\to 0\ \ \textrm{as}\ \
R\to \infty.
\]
Since $s_R\to \infty$ as $R\to \infty$, we conclude that $u\equiv
a_{i_0}$.
\end{proof}

\end{document}